\documentclass[reqno,twoside,a4paper,11pt]{amsart}

\usepackage{amsthm, amssymb, amsmath, amsfonts,eufrak}
\usepackage{url}
\usepackage[colorlinks=true,pdftex]{hyperref}
\hypersetup{linkcolor=red,citecolor=blue}
\usepackage{verbatim}
\usepackage{lineno}
\IfFileExists{epsf.def}{\input epsf.def}{\usepackage{epsf}}

\IfFileExists{myowntimes.sty}{%
\usepackage{mathpazo}
\usepackage{mathrsfs}
\usepackage[notref,notcite]{showkeys}
}
{}
\DeclareFontFamily{OT1}{eusb}{} \DeclareFontShape{OT1}{eusb}{m}{n} {<5> <6> <7> <8> <9> <10> <11> <12> <14.4> eusb10}{}
\DeclareMathAlphabet{\eusb}{OT1}{eusb}{m}{n}

\DeclareFontFamily{OT1}{eusm}{} \DeclareFontShape{OT1}{eusm}{m}{n} {<5> <6> <7> <8> <9> <10> <11> <12> <14.4> eusm10}{}
\DeclareMathAlphabet{\eusm}{OT1}{eusm}{m}{n}

\DeclareFontFamily{OT1}{eufm}{} \DeclareFontShape{OT1}{eufm}{m}{n} {<5> <6> <7> <8> <9> <10> <11> <12> <14.4> eufm10}{}
\DeclareMathAlphabet{\mathfrak}{OT1}{eufm}{m}{n}

\DeclareFontFamily{OT1}{fraktura}{}
\DeclareFontShape{OT1}{fraktura}{m}{n} {<5> <6> <7> <8> <9> <10> <11> <12> <13> <14.4> [1.1] eufm10}{}
\DeclareMathAlphabet{\fraktura}{OT1}{fraktura}{m}{n}

\DeclareFontFamily{OT1}{cmfi}{} \DeclareFontShape{OT1}{cmfi}{m}{n} {<5> <6> <7> <8> <9> <10> <11> <12> <13> <14.4> [0.9] cmfi10}{}
\DeclareMathAlphabet{\cmfi}{OT1}{cmfi}{b}{n}

\DeclareFontFamily{OT1}{cmss}{} \DeclareFontShape{OT1}{cmss}{m}{n} {<5> <6> <7> <8> <9> <10> <11> <12> <13> <14.4> cmss10}{}
\DeclareMathAlphabet{\cmss}{OT1}{cmss}{m}{n}

\newtheoremstyle{thm}{1.8ex}{1.8ex}{\itshape\rmfamily}{} {\bfseries\rmfamily}{}{2ex}{}

\newtheoremstyle{def}{1.8ex}{1.8ex}{\slshape\rmfamily}{} {\bfseries\rmfamily}{}{2ex}{}

\newtheoremstyle{rem}{1.8ex}{1.8ex}{\rmfamily}{} {\bfseries\rmfamily}{}{2ex}{}

\newenvironment{proofsect}[1] {\vspace{0.2cm}\noindent{\rmfamily\itshape#1.}}{\qed\vspace{0.15cm}}

\newcommand\cour[1]{{\fontfamily{pcr}\selectfont #1}}

\theoremstyle{thm}
\newtheorem{theorem}{Theorem}[section]
\newtheorem{lemma}[theorem]{Lemma}
\newtheorem{proposition}[theorem]{Proposition}

\newtheorem*{Main Theorem}{Main Theorem.}
\newtheorem{corollary}[theorem]{Corollary}
\newtheorem*{special theorem}{Lindeberg-Feller Theorem for Martingales}

\theoremstyle{def}

\theoremstyle{rem}
\newtheorem{remark}[theorem]{Remark}

\numberwithin{equation}{section}

\renewcommand{\section}{\secdef\sct\sect}
\newcommand{\sct}[2][default]{%
\refstepcounter{section}
\addcontentsline{toc}{section}{{\tocsection {}{\thesection}{\!\!\!\!#1\dotfill}}{}}
\vspace{0.7cm}
\centerline{\scshape\thesection.\ #1} \nopagebreak \vspace{0.2cm}}
\newcommand{\sect}[1]{%
\vspace{0.4cm} \centerline{\large\scshape\rmfamily #1}
\vspace{0.2cm}}

\renewcommand{\subsection}{\secdef\subsct\sbsect}
\newcommand{\subsct}[2][default]{\refstepcounter{subsection}
\addcontentsline{toc}{subsection}
{{\tocsection{\!\!}{\hspace{1.2em}\thesubsection}{\!\!\!\!#1\dotfill}}{}}
\nopagebreak\vspace{0.45\baselineskip} {\flushleft\bf
\thesubsection~\bf #1.~}
\\*[3mm]\noindent
\nopagebreak}
\newcommand{\sbsect}[1]{\vspace{0.1cm}\noindent
\textbf{#1.~}\vspace{0.1cm}}

\renewcommand{\subsubsection}{%
\secdef \subsubsect\sbsbsect}
\newcommand{\subsubsect}[2][default]{%
\refstepcounter{subsubsection} 
\addcontentsline{toc}{subsubsection}{{\tocsection{\!\!}
{\hspace{3.05em}\thesubsubsection}{\!\!\!\!#1\dotfill}}{}}
\nopagebreak
\vspace{0.15\baselineskip} \nopagebreak {\flushleft\rmfamily
\itshape\thesubsubsection
\ \rmfamily #1\/.}\ }
\newcommand{\sbsbsect}[1]{\vspace{0.1cm}\noindent
\rmfamily \itshape
\arabic{section}.\arabic{subsection}.\arabic{subsubsection} \
\sffamily #1\/.\ }

\renewcommand{\caption}[1]{%
\vglue0.5cm
\refstepcounter{figure}
\begin{minipage}{0.9\textwidth}\small {\sc Figure~\thefigure. }#1\end{minipage}}

\newcommand{\N}        {\mathbb N}
\newcommand{\R}        {\mathbb R}
\newcommand{\B}        {\mathbb B}

\newcommand{\Z}        {\mathbb Z}
\newcommand{\Q}        {\mathbb Q}
\newcommand{\E}        {\mathcal E}

\newcommand{\CC}          {\mathcal{C}} 
\newcommand{\HH}           {\mathcal{H}}
\newcommand{\BB}          {\mathcal{B}}

\newcommand{\GG}          {\mathcal{G}} 
\newcommand{\LL}         {\mathcal{L}}

\newcommand{\twoeqref}[2]{(\ref{#1}--\ref{#2})}

\title[Random Walk among random conductances]{Standard Spectral Dimension for the Polynomial Lower Tail Random Conductances model}
\author[Omar Boukhadra
]{Omar BOUKHADRA}
\begin{document} 
\thanks{\hglue-4.5mm\fontsize{9.6}{9.6}\selectfont\copyright\,2009 by Omar Boukhadra. Provided for non-commercial research and education use. Not for reproduction, distribution or commercial use.\vspace{2mm}\\
\cour{e-mail~: omar.boukhadra@cmi.univ-mrs.fr}}


\maketitle
\vspace{-5mm}
\centerline{\textit{CMI, Universit\'e de Provence};}
\centerline{\textit{D\'epartement de Math\'ematiques, UMC, Constantine, Algeria}}
\vspace{-2mm}
\begin{abstract}
We study models of continuous-time, symmetric, $\Z^{d}$-valued random walks in
random environment, driven
by a field of i.i.d. random nearest-neighbor conductances $\omega_{xy}\in[0,1]$ with a power law with an exponent $\gamma$ near 0. We are interested in estimating the quenched decay of the return probability $P_\omega^{t}(0,0)$, as $t$ tends to $+\infty$.  We show that for $\gamma> \frac{d}{2}$, the standard bound turns out to be of the correct logarithmic order. As an expected consequence, the same result holds for the discrete-time case.
\newline\\
\noindent
{\textbf{keywords}}~:
Markov chains, Random walk, Random environments,
Random conductances, Percolation.
\end{abstract}

\section{\textbf{Introduction}}

We study models of continuous-time random walks in reversible random environment on $\Z^d$, $d\geq2$. Our aim is to derive estimates on the decay of
transition probabilities in the absence of  uniform ellipticity assumption.\\
We derive sharp bounds
on the decay of the quenched return probability when the rates are i.i.d. random variables chosen from a law with
polynomial tail near 0 with exponent $\gamma$. We then prove that the standard bound of the heat-kernel turns out to be of the correct logarithmic order, when $\gamma> d/2$, which follows up recent results of Fontes and Mathieu \cite{Fontes-Mathieu}, Berger, Biskup, Hoffman and Kozma \cite{berger}, and Boukhadra \cite{moi}.

\subsection{Describing the model}
Let us now describe the model more precisely. We consider a family of symmetric, irreducible, nearest-neighbors Markov chains taking their values in $\Z^d$, $d\geq2$, and constructed in the following way. Let $\Omega$ be the set of functions $\omega:\Z^d\times\Z^d \rightarrow \R_{+}$ such that $\omega_{xy}>0$ iff $x \sim y$, and  $\omega_{xy}=\omega_{yx}$ \, ( $x \sim y$ means that $x$ and $y$ are nearest-neighbors). We call elements of $\Omega$ environments.

Define the transition matrix
\begin{equation}
\label{protra}
P_{\omega}(x,y)=\frac{\omega_{xy}}{\pi_{\omega}(x)},
\end{equation}
and the associated Markov generator
\begin{equation}
\label{gg}
(\mathcal{L}^{\omega}f)(x)=\sum_{y\sim x}P_{\omega}(x,y)[f(y)-f(x)].
\end{equation}

$X=\{X_t, t\in \R_+\}$ will be the coordinate process on path space $(\Z^d)^{\R_+}$ and we use the notation $P^{\omega}_{x}$ to denote the unique probability measure on path space
under which $X$  is the Markov process generated by \eqref{gg} and satisfying $X_0=x$, with expectation henceforth denoted by $E^\omega_x$. This process can be described as follows. The moves are those of the discrete time Markov chain with transition matrix given in \eqref{protra} started at $x$, but the jumps occur after independent Poisson~(1) waiting times. Thus, the probability that there have been exactly $i$ jumps at time $t$ is $e^{-t}t^i/i!$ and the probability to be at $y$ after exactly $i$ jumps at time $t$ is $e^{-t}t^iP^i_\omega(x,y)/i!$. 

Since $\omega_{xy}>0$ for all neighboring pairs $(x,y)$, $X_t$ is irreducible under the ``quenched law" $P^\omega_x$ for all $x$. The sum $\pi_\omega(x)=\sum_y\omega_{xy}$ defines an invariant, reversible measure for the corresponding (discrete) continuous-time Markov chain. 


The continuous time semigroup associated with $\LL^\omega$ is defined by
\begin{equation}
\label{sg}
P^\omega_t f(x):=E^\omega_x [f(X_t)].
\end{equation}

Define the heat kernel, that is the kernel of $P^t_\omega$ with respect to $\pi_\omega$, or the transition density of $X_t$, by
\begin{equation}
h^\omega_t(x,y):=\dfrac{P^\omega_t(x,y)}{\pi_\omega(y)}
\end{equation}
Clearly, $h^\omega_t$ is symmetric, that is $h^\omega_t(x,y)=h^\omega_t(y,x)$ and  satisfies the Chapman-Kolmogorov equation as a consequence of the semi-group law $P^\omega_{t+s}=	P^\omega_t P^\omega_s$. We call spectral dimension of $X$ the quantity
\begin{equation}
-2 \lim_{t\rightarrow+\infty}\dfrac{\log h^\omega_t(x,x)}{\log t},
\end{equation}
(if this limit exists).

Such walks under the additional assumptions of uniform ellipticity,
\begin{displaymath}
\exists\alpha>0:\quad \Q(\alpha<\omega_{b}<1/\alpha)=1,
\end{displaymath}
have the standard local-CLT like decay of the heat kernel as proved by Delmotte~\cite{del}:
\begin{equation}
\label{heat-kernel}
 P^{n}_{\omega}(x,y)\leq\frac{c_{1}}{n^{d/2}}\exp\left \{-c_{2}\frac{\vert x-y\vert^{2}}{n}\right\},
\end{equation}
where $c_1,c_2$ are absolute constants.

Once the assumption of uniform ellipticity is relaxed, matters get more complicated. The most-intensely studied example is the simple random walk on the infinite cluster of supercritical bond percolation on~$\Z^d$, $d\ge2$. This corresponds to~$\omega_{xy}\in\{0,1\}$ i.i.d.\ with~$\Q(\omega_b=1)>p_c(d)$ where~$p_c(d)$ is the percolation threshold (cf. \cite{G}). Here an annealed invariance principle has been obtained by De Masi, Ferrari, Goldstein and Wick~\cite{demas1,demas2} in the late 1980s. More recently, Mathieu and Remy~\cite{Mathieu-Remy} proved the on-diagonal (i.e., $x=y$) version of the heat-kernel upper bound \eqref{heat-kernel}---a slightly weaker version of which was also obtained by Heicklen and Hoffman~\cite{Heicklen-Hoffman}---and, soon afterwards, Barlow~\cite{Barlow} proved the full upper and lower bounds on 
$P_\omega^n(x,y)$ of the form \eqref{heat-kernel}. (Both these results hold for $n$ exceeding some random time defined relative to the environment in the vicinity of~$x$ and~$y$). Heat-kernel upper bounds were then used in the proofs of quenched invariance principles by Sidoravicius and Sznitman~\cite{Sidoravicius-Sznitman} for $d\ge4$, and for all $d\ge2$ by Berger and Biskup~\cite{BB} and Mathieu and Piatnitski~\cite{Mathieu-Piatnitski}.

We choose in our case the family $\{ \omega_{b}, b=(x,y),x \sim y, b\in\Z^d\times\Z^d\}$ i.i.d according to a law $\Q$ on $(\R^{\ast}_{+})^{\Z^d}$ such that 
\begin{equation}
\label{1}
\begin{array}{ll}
\omega_{b}\leq 1 & \text{for all } b;\\
\Q(\omega_{b} \leq a)\sim a^{\gamma} & \text{when } a\downarrow 0,
\end{array}
\end{equation}
where $\gamma>0$ is a parameter.

Our work is motivated by the recent study of Fontes and Mathieu~\cite{Fontes-Mathieu} of continuous-time random walks on~$\Z^d$ with conductances given by
\begin{displaymath}
\omega_{xy}=\omega(x)\wedge \omega(y)
\end{displaymath}
for i.i.d.\ random variables~$\omega(x)>0$ satisfying \eqref{1}. For these cases, it was found that the annealed heat-kernel, $\int \text{d}\Q(\omega) P^{\omega}_{0}(X_{t}=0)$, exhibits opposite behaviors, \emph{standard and anomalous}, depending whether  $\gamma\geq d/2$ or $\gamma< d/2$. Explicitly, from (\cite{Fontes-Mathieu}, Theorem 4.3) we have
\begin{equation}
\label{an-log-d}
\int \text{d}\Q(\omega) P^{\omega}_{0}(X_{t}=0)=t^{-(\gamma\wedge\frac{d}{2})+o(1)}, \quad t\rightarrow\infty.
\end{equation}

Further, in a more recent paper~\cite{moi}, we show that the quenched heat-kernel exhibits also opposite behaviors, anomalous and standard, for small and large values of $\gamma$. We first prove for all $d\geq5$ that the return probability shows an anomalous decay that approaches (up to sub-polynomial terms) a random constant times $n^{-2}$ when we push the power $\gamma$ to zero. In contrast, we prove that the heat-kernel decay is as close as we want, in a logarithmic sense, to the standard decay $n^{-d/2}$ for large values of the parameter $\gamma$, i.e.~:  there exists a positive constant $\delta=\delta(\gamma)$ depending only on $d$ and $\gamma$ such that $\Q-a.s.$,  
\begin{equation} 
\label{}
\limsup_{n\rightarrow+\infty}\sup_{x\in\Z^d}\frac{\log P^{n}_{\omega}(0,x)}{\log n}\leq -\frac{d}{2}+\delta(\gamma)\quad \text{and}\quad \delta(\gamma)\xrightarrow[\gamma \to +\infty]{}0,
\end{equation}
These results are a follow up on a paper by Berger, Biskup, Hoffman and
Kozma \cite{berger}, in which the authors proved a universal (non standard) upper
bound for the return probability in a system of random walk among
bounded (from above) random conductances. In the same paper, these
authors supplied examples showing that their bounds are sharp.
Nevertheless, the tails of the distribution near zero in these
examples was very heavy.

\subsection{Main results}
Let $\B^d$ denote the set of unordered nearest-neighbor pairs (i.e., edges) of $\Z^d$ and let $(\omega_b)_{b\in \B^d}$ be i.i.d. random variables with~$(\omega_b)\in\Omega=[0,1]^{\B^d}$. We will refer to $\omega_b$ as the \textit{conductance} of the edge~$b$. The law $\Q$ of the~$\omega$'s will be i.i.d.\ subject to the conditions given in \eqref{1}.

We are interested in estimating the decay of the quenched return probability $P^\omega_{0}(X_t=0)$, as $t$ tends to $+\infty$ for the Markov process associated with the generator defined in \eqref{gg} and we obtain, in the quenched case, a similar result to \eqref{an-log-d}. Although our result is true for all $d\geq 2$, but it is significant when $d\geq 5$.

\smallskip
The main result of this paper is as follows:

\begin{theorem}
\label{th}
For any $\gamma>d/2$, we have
\label{th}
\begin{equation}
\label{log-lim}
\lim_{t\rightarrow+\infty}\frac{\log P^\omega_{0}(X_t=0)}{\log t}=-\frac{d}{2}, \qquad \Q-a.s.
\end{equation}
\end{theorem}

We follow a different approach from the one that Fontes and Mathieu \cite{Fontes-Mathieu} adopted to prove the same result under the annealed law. In the quenched case, the arguments are based on time change, percolation estimates and spectral analysis. Indeed, one operates first a time change to bring up the fact that the random walk viewed only on a strong cluster (i.e. constituted of edges of order $1$) has a standard behavior. Then, we show that the transit time of the random walk in a hole is ``negligible" by bounding the trace of a Markov operator that gives us the Feynman-Kac Formula and this by estimating its spectral gap.

An expected consequence of this Theorem is the following corollary, whose proof is given in part \ref{subsec3} and that gives the same result for the discrete-time case. For the random walk associated with the transition probabilities given in \eqref{protra} for an environment $\omega$ with conductances satisfying the assumption \eqref{1}, we have
\begin{corollary}
For any $\gamma>d/2$, we have
\label{cor}
\begin{equation}
\label{log-lim}
\lim_{n\rightarrow+\infty}\frac{\log P^{2n}_{\omega}(0,0)}{\log n}=-\frac{d}{2}, \qquad \Q-a.s.
\end{equation}
\end{corollary}

\begin{remark}
\label{rem}
 The invariance principle (CLT) (cf. Theorem 1.3 in \cite{QIP}) automatically implies the ``usual'' lower bound on the heat-kernel under weaker conditions on the conductances. Indeed, suppose that $\omega_{xy}\in[0,1]$ and the conductance law is i.i.d. subject to the condition that the probability of $\omega_{xy}>0$ exceeds the threshold for bond percolation on~$\Z^d$, and let $\mathcal{C}_{\infty}$ represents the set of sites that have a path to infinity along bonds with positive conductances. Then, the Markov property and reversibility of~$X$ yield 
	$$
	P^{\omega}_{0}(X_{t}=0)\geq \frac{\pi_\omega(0)}{2d}\sum_{ x\in\mathcal{C}_{\infty}\atop \vert x\vert\leq \sqrt{t}}P^{\omega}_{0}(X_{t/2}=x)^{2}.
	$$
	Cauchy-Schwarz then gives
	$$
	P^{\omega}_{0}(X_{t}=0)\geq 	P^{\omega}_{0}\left(\vert X_{t/2}\vert \leq \sqrt{t}\right)^{2}\frac{\pi_\omega(0)/2d}{ \vert\mathcal{C}_{\infty}\cap[-\sqrt{t},+\sqrt{t}]^{d}\vert}.
	$$

	Now the invariance principle implies that $P^{\omega}_{0}(\vert X_{t/2}\vert \leq \sqrt{t})^{2}$ has a positive limit as~$t\to\infty$ and the Spatial Ergodic Theorem shows that \mbox{$\vert \mathcal{C}_{\infty}\cap [-\sqrt{t},+\sqrt{t}]^{d}\vert$} grows proportionally to~$t^{d/2}$. Hence we get 
	$$
	P ^{\omega}_{0}(X_{t}=0)\geq \frac{C(\omega)}{t^{d/2}},
	$$
	with~$C(\omega)>0$  a.s. on the set~$\{0\in \mathcal{C}_{\infty}\}$ and $t$ large enough. Note that, in $d=2,3$, this complements nicely the ``universal'' upper bounds derived in~\cite {berger}, Theorem 2.1. Thus, for $d\geq2$, we have
	\begin{equation}
	\label{tc}
	\liminf_{t\rightarrow+\infty}\frac{\log P^\omega_{0}(X_t=0)}{\log t}\geq-\frac{d}{2} \qquad \Q-a.s.
	\end{equation}
	and for the cases $d=2,3,4$, we have already the limit \eqref{log-lim} under weaker conditions on the conductances (see \cite{berger}, Theorem~2.2). So, under assumption~\eqref{1}, it remains to study the cases where $d\geq5$ and prove that for $\gamma>d/2$,
\begin{equation}
	\label{tc2}
	\limsup_{t\rightarrow+\infty}\frac{\log P^\omega_{0}(X_t=0)}{\log t}\leq-\frac{d}{2} \qquad \Q-a.s.
	\end{equation}	
\end{remark}

\section{\textbf{A time changed process}}

In this section, we introduce a time changed process $\hat{X}$.

Choose a threshold parameter $\xi>0$ such that $\Q(\omega_b\geq \xi)>p_c(d)$. The i.i.d.\ nature of the measure~$\Q$ ensures that for $\Q$ almost any environment $\omega$, the percolation graph $(\Z^d,\{e\in \B^d; \omega_b\geq \xi\})$ has a unique infinite cluster that we denote with $\CC^\xi(\omega)$.

We will refer to the connected components of
the complement of $\CC^\xi(\omega)$ in $\Z^d$ as {\it holes}. By definition, holes are connected sub-graphs of the grid. The sites which belong to $\CC^\xi(\omega)$ are all the endvertices of its edges. The other sites belong to the holes. Let $\HH^\xi(\omega)$ be the collection of all holes. Note that holes may contain edges such that $\omega_b\geq\xi$.

First, we will give some important characterization of the volume of the holes, see Lemma 5.2 in \cite{QIP}. $\CC$ denotes the {\it infinite cluster}.
\begin{lemma}
\label{v-h}
There exists $\bar{p}<1$ such that for $p>\bar{p}$, for almost any realization of bond
percolation of parameter~$p$ and for large enough $n$, any connected component of the complement
of the infinite cluster $\CC$ that intersects the box $[-n,n]^d$ has volume smaller than $(\log n)^{5/2}$.
\end{lemma}
Now, choose $\xi>0$ such that $\Q(\omega_b\geq \xi)>\bar{p}$, and let $\widehat\CC(\omega)$ denote the associated {\it infinite cluster}.

Define the conditioned measure
$$
\hat{\Q}_0(\,\cdot\,)=\Q(\,\cdot\,|0\in\widehat{\CC}(\omega)).
$$ 
Consider the following additive functional of the random walk~:
$$
\hat{A}(t)=\int^t_0\text{\textbf{1}}_{\{X(s)\in\widehat{\CC}(\omega)\}} ds,
$$
its inverse $(\hat{A})^{-1}(t)=\inf\{s; \hat{A}(s)>t\}$ and define the corresponding time changed process
$$
\hat{X}(t)=X((\hat{A})^{-1}(t)).
$$

Thus the process $\hat{X}$ is obtained by suppressing in the trajectory of $X$ all the visits to the
holes. Note that, unlike $X$, the process $\hat{X}$ may perform long jumps when straddling holes.

As $X$ performs the random walk in the environment~$\omega$, the behavior of the random process $\hat{X}$ is described in the next
\begin{proposition}
Assume that the origin belongs to $\widehat{\CC}(\omega)$. Then, under $P^\omega_0$, the random process $\hat{X}$ is a symmetric Markov process on $\widehat{\CC}(\omega)$.
\end{proposition}

The Markov property, which is not difficult to prove, follows from a very general argument
about time changed Markov processes. The reversibility of $\hat{X}$ is a consequence of the
reversibility of $X$ itself as will be discussed after equation~\eqref{Xf-tr}.

The generator of the process $\hat{X}$ has the form
\begin{equation}
\label{Xf}
\hat{\LL}^{\omega}f(x)=\frac{1}{\eta^\omega(x)}\sum_y\hat{\omega}(x,y)(f(y)-f(x)),
\end{equation}
where
\begin{eqnarray}
\label{Xf-tr}
\frac{\hat{\omega}(x,y)}{\eta^\omega(x)}
&=&
\lim_{t\rightarrow 0}\frac{1}{t}P^\omega_x(\hat{X}(t)=y)\nonumber
\\
&=&
P^\omega_x(y\, \small{\text{is the next point in}\, \widehat{\CC}(\omega)\, \text{visited by the random walk}}),
\end{eqnarray}
if both $x$ and $y$ belong to $\widehat{\CC}(\omega)$ and $\hat{\omega}(x,y)=0$ otherwise.

The function $\hat{\omega}$ is symmetric~: $\hat{\omega}(x,y)=\hat{\omega}(y,x)$ as follows from the reversibility of $X$ and
formula \eqref{Xf-tr}, but it is no longer of nearest-neighbor type i.e. it might happen that $\hat{\omega}(x,y)\neq0$ although $x$ and $y$ are not neighbors. More precisely, one has the following picture~: $\hat{\omega}(x,y)=0$ unless either $x$ and $y$ are neighbors and $\omega(x,y)\geq\hat{\xi}$, or there exists a hole, $h$, such that both $x$ and $y$ have neighbors in $h$. (Both conditions may be fulfilled by the same pair $(x,y)$.)

Consider a pair of neighboring points $x$ and $y$, both of them belonging to the infinite cluster $\widehat{\CC}(\omega)$ and such that $\omega(x,y)\geq\hat{\xi}$, then
\begin{equation}
\label{b-f}
\hat{\omega}(x,y)\geq \xi.
\end{equation} 
This simple remark will play an important role. It implies, in a sense
that the parts of the trajectory of $\hat{X}$ that consist in nearest-neighbors jumps are similar to
what the simple symmetric random walk on $\widehat{\CC}(\omega)$ does. Precisely, we will need the following important fact  that $\hat{X}$ obeys the standard heat-kernel bound~:
\begin{lemma}
\label{sb-f}
$\hat{\Q}_0-a.s.$ there exists a random variable $C(\omega,x)$ such that for large enough $t$, we have
\begin{equation}
P^\omega_x(\hat{X}(t)=y)\leq \frac{C(\omega,x)}{t^{d/2}},
\end{equation}
for every $x\in\widehat{\CC}(\omega)$ and $y\in\Z^d$.
\end{lemma}
For a proof, we refer to \cite{QIP}, Lemma 4.1. In the discrete-time case, see \cite{berger}, Lemma 3.2.

\section{\textbf{Proof of Theorem \ref{th} and Corollary \ref{cor}}}

The upper bound \eqref{tc2} will be discussed in part \ref{subsec2} and the proof of Corollary \ref{cor} is given in part \ref{subsec3}. We first start with some preliminary lemmata.
\subsection{Preliminaries}

First, let us recall the following standard fact from
Markov chain theory~:
\begin{lemma}
\label{D}
The function $t\longmapsto P^\omega_0(X_t=0)$ is non increasing.
\end{lemma}
\begin{proof}
Let $\left\langle f,g\right\rangle_\omega$ denote a scalar product in $L^2(\pi_\omega)=L^2(\Z^d,\pi_\omega)$. Then
\begin{equation}
\pi_\omega(0)P^\omega_0(X_t=0)=\left\langle \delta_0,P^\omega_t\delta_0\right\rangle_\omega
\end{equation} 
Since $P^\omega_t$ is self-adjoint for all $t\geq0$ and $\Vert P^\omega_t\Vert_2\leq 1$, the function $P^\omega_0(X_t=0)$ is non increasing.
\end{proof} 

Next, Let $B_N=[-N,N]^d$ be the box centered at the origin and of radius $N$ that we choose as a function of time such that $t\approx N^2(\log N)^{-b}$, with $b>1$, and let $\BB_{N}$ denote the set of nearest-neighbor bonds of $B_N$, i.e., $\BB_{N}=\{b=(x,y): x,y\in B_N, x\sim y\}$. We have
\begin{lemma}
\label{L}
Under assumption~\eqref{1},
\begin{equation}
\label{LL}
\lim_{N\rightarrow+\infty}\frac{\log\inf_{b\in\BB_N}\omega_b}{\log N}=-\frac{d}{\gamma},\qquad \Q-a.s.
\end{equation}
\end{lemma}
Thus, for arbitrary $\mu>0$, we can write $\Q-a.s.$ for $N$ large enough,
\begin{equation}
\label{mu}
\inf_{b\in\BB_N}\omega_b\geq N^{-(\frac{d}{\gamma}+\mu)}.
\end{equation}
For a proof of this standard fact, see \cite{Fontes-Mathieu}, Lemma~3.6.

%

Consider now the following formula
\begin{equation}
\label{F-K}
R^\omega_t f(x)=E^\omega_x \left[f(X_t)e^{-\lambda\hat{A}(t)}\right], \qquad t\geq 0, \lambda\geq0,\, x\in\Z^d
\end{equation}
This object will play a key role in our proof of Theorem \ref{th}. Let $L^2_b(\pi_\omega)$ denote the set of bounded functions of $L^2(\pi_\omega)$. We have

\begin{proposition}
$R=\{R^\omega_t, t\geq0\}$ defines a  semigroup (of symmetric operators) on $L^2_b(\pi_\omega)$, with generator
\begin{equation}
\label{ggg}
\GG^\omega f=\LL^\omega f-\lambda\varphi f
\end{equation}
where $\varphi=\text{\textbf{1}}_{\{\cdot\,\in\widehat{\CC}(\omega)\}}$. One also has the perturbation identities
\begin{equation}
\label{pr}
\begin{split}
&  R^\omega_t f(x)=P^\omega_t f(x)-\lambda\int^t_0 P^\omega_t(\varphi R^\omega_{t-s}f)(x)ds\\
& \qquad\quad\,\, =P^\omega_t f(x)-\lambda\int^t_0 R^\omega_s(\varphi P^\omega_{t-s}f)(x)ds,\\
& t\geq0, x\in\Z^d, f\in L^2_b(\pi_\omega).
\end{split}
\end{equation}
\end{proposition}
\begin{proof}
The proof, that we give here, very closely mimics the arguments of~\cite{Sznitman}, Theorem~1.1. Indeed, we begin with the proof of \eqref{pr}. Observe that $P^\omega_x-a.s.$, for every $x\in\Z^d$, the time function $t\mapsto e^{-\lambda\hat{A}(t)}$ is continuous, and

$$
\lim_{h\rightarrow0}\frac{1}{h}\left(e^{-\lambda\hat{A}(s+h)}-e^{-\lambda\hat{A}(s)}\right)=-\lambda\varphi(X_s)e^{-\lambda\hat{A}(s)},\quad \forall s\in[0,t],\,t>0,
$$
except possibly for a countable set $\{\alpha_i\}_{i\in I}\subset (0,t], \vert I\vert\subset\N^*$. 
Then, for $t\geq0$, we have
\begin{eqnarray}
\label{sz}
e^{-\lambda\hat{A}(t)}
&=&
1-\lambda\int^t_0\varphi(X_s)\exp\left\{-\int^t_s\lambda\varphi(X_u)du\right\}ds\nonumber
\\
&=&
1-\lambda\int^t_0\varphi(X_s)e^{-\lambda\hat{A}(s)}ds.
\end{eqnarray}
Multiplying both members of the first equality of~\eqref{sz} by $f(X_t)$ and integrating we find~:
\begin{eqnarray*}
\label{fl}
R^\omega_t f(x)
&=&
P^\omega_t f(x)-\lambda\int^t_0 E^\omega_x\left[\varphi(X_s)\exp\left\{-\lambda\int^t_s\varphi(X_u)du\right\}f(X_t)\right]ds\nonumber
\\
&=&
P^\omega_t f(x)-\lambda\int^t_0 P^\omega_s(\varphi R^\omega_{t-s}f)(x)ds,\quad \text{(Markov property)},
\end{eqnarray*}
\smallskip

\noindent
which is the first identity of \eqref{pr}. Analogously  we find the second identity of \eqref{pr} with the help of the second line of \eqref{sz}. This completes the proof of \eqref{pr}.

By direct inspection of the first line of \eqref{pr}, using dominated convergence, the fact that for $s>0$ , $P^\omega_s$ maps every function from $L^2_b(\pi_\omega)$ into function in $L^2_b(\pi_\omega)$, and the inequality $|R^\omega_t f|\leq c(t,\varphi)|P^\omega_t f|$, it is easy to argue that $R^\omega_t f\in L^2_b(\pi_\omega)$, for $f\in L^2_b(\pi_\omega)$. Moreover for $s,t\geq0$,
\begin{eqnarray*}
R^\omega_{s+t} f(x)
&=&
E^\omega_x\left[e^{-\lambda\hat{A}(s)}\exp\left\{-\lambda\int^{s+t}_s\varphi(X_u)du\right\}f(X_t)\right]
\\
&=&
R^\omega_s(R^\omega_tf)(x),\qquad \text{(Markov property)}.
\end{eqnarray*}
\smallskip

\noindent
We thus proved that $R^\omega_t$ defines a semigroup on $L^2_b(\pi_\omega)$. The strong continuity of this semigroup follows readily by letting $t$ tend to $0$ in \eqref{pr}.

Let us finally prove~\eqref{ggg}. To this end notice that for $f\in L^2_b(\pi_\omega)$~:
$$
\frac{1}{t}\int^t_0 P^\omega_t(\varphi R^\omega_{t-s}f)(x)ds\xrightarrow[t \to 0]{}\varphi (x)f(x),\quad\text{uniformly in}\,\,x,
$$
since $\varphi(X_t)\mapsto\varphi(x)$, $P^\omega_x-a.s.$ Coming back to the first line of \eqref{pr}, this proves that the convergence of $\frac{1}{t}(R^\omega_t-f)$ or $\frac{1}{t}(P^\omega_t-f)$ as $t$ tends to $0$, are equivalent and \eqref{ggg} holds.
\end{proof}

Let $\LL^\omega_N$ and $\GG^{\omega}_N$ be respectively the restrictions of the operators $\LL^\omega$ and $\GG^{\omega}$ (cf. \twoeqref{gg}{ggg}) to the set of functions on $B_N$ with Dirichlet boundary conditions outside $B_N$, that we denote by $L^2(B_N, \pi_\omega)$ (that is, $\LL^\omega_N$ and $\GG^{\omega}_N$ are respectively the generators of the process $X$ and 
of the semigroup $R$,  which coincide with the ones given by $\LL^\omega$ and $\GG^{\omega}$ until the process $X$
leaves $B_N$ for the first time, and then it is killed). Then $-\LL^\omega_N$ and $-\GG^{\omega}_N$ are  positive symmetric
operators and we have
$$
\GG^\omega_N f=\LL^\omega_N f-\lambda\varphi f,
$$
with associated semigroup defined by 
$$
(R^{\omega,N}_tf)(0)
:=
E^\omega_0\left[f(X_t)e^{-\lambda\hat{A}(t)};t<\tau_N\right].
$$
where $\tau_N$ is the exit time for the process $X$ from the box $B_N$.

Let $\{\Lambda^\omega_i(B_N), i\in[1,\#B_N]\}$ be the set of eigenvalues of~$-\GG^{\omega}_N$ labeled
in increasing order, and $\{\psi^\omega_i, i\in[1,\#B_N]\}$ the corresponding eigenfunctions with due normalization in $L^2(B_N,\pi_\omega)$.

\subsection{Proof of the upper bound}
\label{subsec2}
In this last part, we will complete the proof of Theorem~\ref{th} by giving the proof of the upper bound \eqref{tc2}.

\begin{proofsect}{Proof of Theorem~\ref{th}}

Assume that the origin belongs to $\widehat{\CC}(\omega)$. By lemma~\ref{D}, we have
$$
P^\omega_0(X(t)=0)\leq \frac{2}{t}\int^{t}_{t/2}P^\omega_0(X(s)=0)\text ds=\frac{2}{t}E^\omega_0\left[\int^{t}_{t/2}\text{\textbf{1}}_{\{X(s)=0\}}\text ds\right].
$$
The additive functional $\hat{A}$ being a continuous increasing function of the time and null outside the support of the measure $\text{d}\hat{A}(s)$, so by operating a variable change by setting  $s=(\hat{A})^{-1}(u)$ (i.e. $u=\hat{A}(s)$), we get
\begin{eqnarray*}
E^\omega_0\left[\int^{t}_{t/2}\text{\textbf{1}}_{\{X(s)=0\}}\text ds\right]
&=&
E^\omega_0\left[\int^{t}_{t/2}\text{\textbf{1}}_{\{X(s)=0\}}\varphi(X(s))\text ds\right]
\\
&=&
E^\omega_0\left[\int^{\hat{A}(t)}_{\hat{A}(t/2)}\text{\textbf{1}}_{\{\hat{X}(u)=0\}}\text du\right],
\end{eqnarray*}
which is bounded by
$$
E^\omega_0\left[\int^{t}_{\hat{A}(t/2)}\text{\textbf{1}}_{\{\hat{X}(u)=0\}}\text du\right],
$$
since $\hat{A}(t)\leq t$.

Therefore, for $\epsilon\in(0,1)$
\begin{eqnarray*}
\begin{split}
& P^\omega_0(X(t)=0)
\leq \frac{2}{t}E^\omega_0\left[\int^{t}_{\hat{A}(t/2)}\text{\textbf{1}}_{\{\hat{A}(t/2)\geq t^\epsilon\}}\text{\textbf{1}}_{\{\hat{X}(u)=0\}}\text du\right]\\
& \qquad\qquad \qquad \qquad\qquad+ \frac{2}{t}E^\omega_0\left[\int^{t}_{\hat{A}(t/2)}\text{\textbf{1}}_{\{\hat{A}(t/2)\leq t^\epsilon\}}\text{\textbf{1}}_{\{\hat{X}(u)=0\}}\text du\right]\\
& \qquad \qquad\qquad\,\leq 
\frac{2}{t}\int^{t}_{t^\epsilon}P^\omega_0(\hat{X}(u)=0)\text du
+  \frac{2}{t} \int^{t}_{0}P^\omega_0(\hat{A}(t/2)\leq t^\epsilon)\text du
\end{split}
\end{eqnarray*}
and using lemma~\ref{sb-f},
\begin{eqnarray}
\label{pb}
P^\omega_0(X(t)=0)
&\leq&
\frac{C}{t}\int^{t}_{t^\epsilon}u^{-d/2}\text du +\frac{2}{t} P^\omega_0(\hat{A}(t/2)\leq t^\epsilon) t \nonumber
\\
&\leq&
\frac{C}{t^{\epsilon\frac{d}{2}-\epsilon+1}}+2 P^\omega_0(\hat{A}(t/2)\leq t^\epsilon)
\end{eqnarray}

It remains to estimate the second term in the right-hand side of the last inequality,  i.e. $P^\omega_0(\hat{A}(t/2)\leq t^\epsilon)$ or more simply $P^\omega_0(\hat{A}(t)\leq 2^\varepsilon t^\epsilon)$, but we can neglect the constant $2^\epsilon$ in the calculus as one will see in \eqref{2en}. 

For each $\lambda\geq0$, Chebychev inequality gives

\begin{eqnarray}
\label{i11}
P^\omega_0(\hat{A}(t)\leq t^\epsilon)
&=&
P^\omega_0(\hat{A}(t)\leq t^\epsilon;t<\tau_N)+P^\omega_0(\hat{A}(t)\leq t^\epsilon;\tau_N\leq t)\nonumber
\\
&\leq&
P^\omega_0\left(e^{-\lambda\hat{A}(t)}\geq e^{-\lambda t^\epsilon};t<\tau_N\right)+P^\omega_0(\tau_N\leq t)\nonumber
\\
&\leq&
e^{\lambda t^\epsilon}E^\omega_0\left[e^{-\lambda\hat{A}(t)};t<\tau_N\right]+P^\omega_0(\tau_N\leq t).
\end{eqnarray}

From the Carne-Varopoulos inequality, it follows that

\begin{equation}
\label{c-v}
P^\omega_0(\tau_N\leq t)\leq CtN^{d-1}e^{-\frac{N^2}{4t}}+e^{-ct},
\end{equation}
where $C$ and $c$ are numerical constants, see Appendix C in \cite{Mathieu-Remy}. With our choice of $N$ such that $t\approx N^2(\log N)^{-b}$ ($b>1$), we get that $P^\omega_0(\tau_N\leq t)$ decays faster than any polynomial as $t$ tends to $+\infty$.

Thus Theorem~\ref{th} will be proved if we can check, for a particular choice of~$\lambda>0$ that may depends on $t$, that

\begin{equation}
\label{t-e}
\limsup_{t\rightarrow +\infty}\frac{\log \left(e^{\lambda t^\epsilon}E^\omega_0\left[e^{-\lambda\hat{A}(t)};t<\tau_N\right]\right)}{\log t}\leq-\frac{d}{2}.
\end{equation}
That will be true if $e^{\lambda t^\epsilon}E^\omega_0\left[e^{-\lambda\hat{A}(t)};t<\tau_N\right]$ decays faster than any polynomial in $t$ as $t$ tends to $+\infty$.

 The Dirichlet form of $-\LL^\omega_N$ on $L^2(B_N,\pi_\omega)$ endowed with the usual scalar product, can be written as
 
$$
\E^{\omega,N}(f,f)=\frac{1}{2}\sum_{b\in\BB_{N+1}}(df(b))^2\omega_b,
$$
where $df(b)=f(y)-f(x)$ and the sum ranges over $b=(x,y)\in\BB^\omega_{N+1}$. By the min-max Theorem (see \cite{hj}) and \eqref{ggg}, we have

\begin{equation}
\label{V2}
\Lambda^\omega_1(B_N)=\inf_ {f\not\equiv0}\frac{\E^{\omega,N}(f,f)+\lambda\sum_{x\in\widehat{\CC}_N}f^2(x)\pi_\omega(x)}{\pi_\omega(f^2)}.
\end{equation}
where $\widehat{\CC}_N$ is the largest connected component of $\widehat{\CC}(\omega)\cap B_N$, and the infimum is taken over functions with Dirichlet boundary conditions.  (Recall that $\Lambda^\omega_1(B_N)$ is the first eigenvalue of $-\GG^{\omega}_N$ and for notational ease, we simply use $\Lambda^\omega_1$.)

To estimate the decay of the first term in the right-hand side of \eqref{i11}, we will also need to estimate the first eigenvalue~$\Lambda^\omega_1$. Recall that $\mu$ denotes an arbitrary positive constant.

\begin{lemma}
\label{vph}
Under assumption \eqref{1}, for any $d\geq2$ and $\gamma>0$, we have $\Q-a.s.$ for $N$ large enough,
\begin{equation}
\label{vphe}
\Lambda^\omega_1\geq (8d)^{-1}N^{-(\frac{d}{\gamma}+\mu)} (\log n)^{-5},
\end{equation}
 for $\lambda$ proportional to $N^{-(\frac{d}{\gamma}+\mu)}$.
\end{lemma}
\begin{proof}
For some arbitrary $\mu>0$, let  $N$ be large enough so that \eqref{mu} holds. Let $h$ be a hole that intersects the box $B_N$, and  for notational ease we will use the same notation for $h\cap B_N$. Define $\partial h$ to be the outer boundary of $h$, i.e. the set of sites in $\widehat{\CC}_N$ which are adjacent to some vertex in $h$. Let us associate to each hole $h$ a fixed site~$h^*\in\widehat{\CC}_N$ situated at the outer boundary of $h$ and  for $x\in h$ call $\kappa(x,h^*)$ a self-avoiding path included in $h$ with end points $x$ and $h^*$, and let $\vert\kappa(x,h^*)\vert$ denote the length of such a path.

Now let $f\in L^2(B_N,\pi_\omega)$ and let $\BB_\omega(h)$ denote the set of the bonds of $h$. For each $x\in h$, write
$$
f(x)=\sum_{b\in\kappa(x,h^*)}df(b)+f(h^*)
$$
and, using Cauchy-Schwarz
$$
 f^2(x)\leq 2\vert \kappa(x,h^*)\vert \sum_{b\in\kappa(x,h^*)}\vert df(b)\vert^2+ 2 f^2(h^*).
$$
In every path $\kappa(x,h^*)$, we see each bond only one time. Multiply the last inequality by $\pi_\omega(x)$ and sum over $x\in h$ to obtain 
\begin{equation}
\label{}
\begin{split}
&\sum_{x\in h} f^2(x)\pi_\omega(x)\\
&\quad\leq2\sum_{x\in h}\vert \kappa(x,h^*)\vert \sum_{b\in\kappa(x,h^*)}\vert df(b)\vert^2\pi_\omega(x)+2\sum_{x\in h}f^2(h^*)\pi_\omega(x)
\\
&\quad \leq 4d\max_{x\in h}\vert\kappa(x,h^*)\vert\max_{b\in\BB_\omega(h)}\frac{1}{\omega_b}\,\#h\sum_{b\in\BB_\omega(h)}\vert df(b)\vert^2\omega_b \\
&\qquad\qquad\qquad\qquad\qquad\qquad+2\sum_{x\in h}f^2(h^*)\pi_\omega(x),
\end{split}
\end{equation}
which, by virtue of lemma \ref{v-h}, \eqref{1}, \eqref{mu} and since $\pi_\omega(h^*)\geq \xi$, is bounded  by
$$
4dN^{\frac{d}{\gamma}+\mu}(\log N)^{5}\sum_{b\in\BB_\omega(h)}\vert df(b)\vert^2\omega_b+\frac{4d}{\xi} \#h f^2(h^*)\pi_\omega(h^*),
$$
Thus, 

$$
\sum_{x\in h}f^2(x)\pi_\omega(x)\leq 4dN^{\frac{d}{\gamma}+\mu}(\log N)^{5}\sum_{b\in\BB_\omega(h)}\vert df(b)\vert^2\omega_b+\frac{4d}{\xi} \#hf^2(h^*)\pi_\omega(h^*).
$$

Let $\widehat{\CC}^c_N=\widehat{\CC}^c_N(\omega)$ denote the complement of $\widehat{\CC}_N(\omega)$ in the box $B_N$ and sum over $h$ to obtain

$$
\sum_{x\in \widehat{\CC}^c_N}f^2(x)\pi_\omega(x)\leq 8dN^{\frac{d}{\gamma}+\mu}(\log N)^{5}\E^{\omega,N}(f,f)+ \frac{8d^2}{\xi}\#h\sum_{x\in \widehat{\CC}_N}f^2(x)\pi_\omega(x),
$$
where in the last term, we multiply by $2d$ since  we may associate the same $h^*$ to $2d$ different holes.
Then

\begin{equation}
\label{}
\begin{split}
& \sum_{x\in B_N} f^2(x)\pi_\omega(x)\\
& \quad\leq (1+8d^2 (\log N)^{5/2} \xi^{-1})\sum_{x\in\widehat{\CC}_N}f^2(x)\pi_\omega(x)+ 8dN^{\frac{d}{\gamma}+\mu}(\log N)^{5}\E^{\omega,N}(f,f).\nonumber\\
& \quad\leq 8d^2 (\log N)^{5} \xi^{-1}\sum_{x\in\widehat{\CC}_N}f^2(x)\pi_\omega(x)+ 8dN^{\frac{d}{\gamma}+\mu}(\log N)^{5}\E^{\omega,N}(f,f).\nonumber
\end{split}
\end{equation}
So, according to \eqref{V2} and for $\lambda=d \xi^{-1} N^{-(\frac{d}{\gamma}+\mu)}$, we get
\begin{equation}
\label{V2B}
\Lambda^\omega_1\geq (8d)^{-1}N^{-(\frac{d}{\gamma}+\mu)} (\log N)^{-5}.
\end{equation}
\end{proof}

Let us get back to the proof of the upper bound. Let 
$$
\lambda=d \xi^{-1} N^{-(\frac{d}{\gamma}+\mu)}; \quad m(N):= (8d)^{-1}N^{-(\frac{d}{\gamma}+\mu)}(\log N)^{-5}.
$$ 
For $f\equiv\textbf{1}$, observe that

\begin{eqnarray*}
(R^{\omega,N}_tf)(0)
&=&
E^\omega_0\left[e^{-\lambda\hat{A}(t)};t<\tau_N\right]
\\
&=&\sum_ie^{-\Lambda^\omega_it}\left\langle \textbf{1}, \psi^\omega_i\right\rangle\psi^\omega_i(0),
\end{eqnarray*}
and
\begin{eqnarray*}
(R^{\omega,N}_tf)^2(0)\pi_\omega(0)
&\leq&
\sum_{x\in B_N}(R^{\omega,N}_tf)^2(x)\pi_\omega(x)
\\
&=&
\sum_ie^{-2\Lambda^\omega_it}\left\langle \textbf{1}, \psi^\omega_i\right\rangle^2
\\
&\leq&
e^{-2\Lambda^\omega_1t}\sum_x\textbf{1}^2(x)\pi_\omega(x)
\\
&\leq&
2^{d+2} d N^d e^{-2\Lambda^\omega_1t}.
\end{eqnarray*}

Then, for large enough $t$ and by \eqref{V2B}, we have

\begin{eqnarray}
\label{2en}
e^{\lambda t^\epsilon}E^\omega_0\left[e^{-\lambda\hat{A}(t)};t<\tau_N\right]\nonumber
&\leq &
(2^{d+2} d)^{1/2} e^{\lambda t^\epsilon}e^{-t\Lambda^\omega_1}N^{d/2}\nonumber
\\
&\leq &
(2^{d+2} d)^{1/2} N^{d/2} \exp\left\{\lambda t^\epsilon-tm(N)\right\}\nonumber
\\
&\leq&
(2^{d+2} d)^{1/2} N^{d/2} e^{-\frac{t}{2} m(N)},
\end{eqnarray}
since $\epsilon<1$.
\smallskip

By our choice of $t\approx N^2(\log N)^{-b}$ ($b>1$), we deduce
\begin{eqnarray}
\label{dd}
e^{\lambda t^\epsilon}E^\omega_0\left[e^{-\lambda\hat{A}(t)};t<\tau_N\right]
&\leq&
(2^{d+2} d)^{1/2} N^{d/2}\exp\left\{-[16d (\log N)^{b+5}]^{-1}N^2 N^{-(\frac{d}{\gamma}+\mu)}\right\}\nonumber
\\
&\ll&
t^{-\frac{d}{2}} \qquad \text{if}\quad \gamma>\frac{d}{2-\mu},
\end{eqnarray}
which yields \eqref{t-e}.

In conclusion, as $\mu$ is arbitrary and according to \twoeqref{c-v}{t-e}, we obtain  
$$
\limsup_{t\rightarrow+\infty}\frac{\log P^\omega_0(\hat{A}(t)\leq t^\epsilon)}{\log t}\leq-\frac{d}{2} \qquad \text{for}\quad \gamma>\frac{d}{2},
$$
and finally, by \eqref{pb}
$$
\lim_{\epsilon\rightarrow1}\limsup_{t\rightarrow+\infty}\frac{\log P^\omega_0(X(t)=0)}{\log t}\leq -\frac{d}{2}\qquad \text{for}\quad \gamma>\frac{d}{2},
$$
which gives \eqref{tc2}.

We conclude that for any sufficiently small $\hat{\xi}$, then $\hat{\Q}_0-$a.s. \eqref{log-lim} is true and since $\Q\left(\cup_{\xi>0}\{0\in \CC^\xi(\omega)\}\right)=1$, it remains true $\Q$-a.s.
\end{proofsect}

\subsection{\textbf{Proof of the discrete-time case}}
\label{subsec3}
\begin{proof}
In the same way, the lower bound holds by the Invariance Principle (cf. \cite{BP}) and the Spatial Ergodic Theorem (see Remark \ref{rem}). \\
For the upper bound, let $(N_t)_{t\geq0}$ be a Poisson process of rate $1$. Set $n=\left\lfloor t\right\rfloor$. $P^{2n}_\omega(0,0)$ being a non increasing function of $n$ (cf. Lemma \ref{D}), then
\begin{eqnarray*}
P^\omega_0(X_t=0)
&=&
e^{-t}\sum_{k\geq0}\frac{t^k}{k!}P^k_\omega(0,0)
\\
&\geq&
e^{-t}\sum^{2n}_{k=0}\frac{t^k}{k!}P^k_\omega(0,0)
\\
&\geq&
P^{2n}_\omega(0,0)\left[e^{-t}\sum^{2n}_{k=0}\frac{t^k}{k!}\right]
\\
&=&
P^{2n}_\omega(0,0)\text{Prob}(N_t\leq 2n).
\end{eqnarray*}
By virtue of the LLN, we have
$$
\text{Prob}(N_t\leq 2n)\xrightarrow[t \to +\infty]{}1.
$$ 
From here the claim follows.
\end{proof}

\section*{Acknowledgments}
\noindent 
I would like to thank my Ph.D. advisor, Pierre Mathieu for suggestions and discussions on this problem. I thank my father Youcef Bey.

\end{document}